\documentclass[10pt]{amsart}

\usepackage[latin1]{inputenc}
\usepackage[english]{babel}
\usepackage{indentfirst}
\usepackage{amssymb}
\usepackage{amsthm}
\usepackage[]{graphicx}
\usepackage[all]{xy}

\newcommand{\impli}{\Rightarrow}
\newcommand{\N}{\mathbb{N}}
\newcommand{\Nat}{\mathbb{N}}

\newcommand{\cS}{\mathcal{S}}
\newcommand{\cG}{\mathcal{G}}
\newcommand{\cH}{\mathcal{H}}
\def\epsilon{\varepsilon}

\newcommand{\sub}{\subseteq}

\newtheorem{theo}{Theorem}[section]
\newtheorem{lem}[theo]{Lemma}
\newtheorem{pro}[theo]{Proposition}
\newtheorem{cor}[theo]{Corollary}

\newtheorem{defi}[theo]{Definition}
\newtheorem{rem}[theo]{Remark}

\newtheorem{exa}[theo]{Example}

\numberwithin{equation}{section}

\author{Jos\'e Rodr\'{i}guez}
\address{Dpto. de Ingenier\'{i}a y Tecnolog\'{i}a de Computadores,
Facultad de Inform\'{a}tica, Universidad de Murcia, 30100 Espinardo (Murcia), Spain}
\email{joserr@um.es}



\thanks{Research supported by projects MTM2014-54182-P and MTM2017-86182-P (AEI/FEDER, UE)}

\title{Ces\`{a}ro convergent sequences in the Mackey topology}

\begin{document}

\begin{abstract}
A Banach space $X$ is said to have property~($\mu^s$) if every weak$^*$-null sequence in~$X^*$
admits a subsequence such that all of its subsequences are Ces\`{a}ro convergent to~$0$ with respect to the Mackey topology.
This is stronger than the so-called property~(K) of Kwapie\'{n}. We prove that property~$(\mu^s)$ holds for 
every subspace of a Banach space which is strongly generated by an operator with Banach-Saks adjoint
(e.g. a strongly super weakly compactly generated space). 
The stability of property~$(\mu^s)$ under $\ell^p$-sums is discussed.
For a family $\mathcal{A}$ of relatively weakly compact subsets of~$X$, we consider the weaker property~$(\mu_\mathcal{A}^s)$ which 
only requires uniform convergence on the elements of~$\mathcal{A}$, and we give some applications to Banach lattices
and Lebesgue-Bochner spaces. We show that every Banach lattice with order continuous norm and weak unit has property $(\mu_\mathcal{A}^s)$ for the family of all
$L$-weakly compact sets. This sharpens a result of de Pagter, Dodds and Sukochev.
On the other hand, we prove that $L^1(\nu,X)$ (for a finite measure~$\nu$) has property~$(\mu_\mathcal{A}^s)$ for the family 
of all $\delta\mathcal{S}$-sets whenever $X$ is a subspace of a strongly super weakly compactly generated space. 
\end{abstract}

\maketitle

\section{Introduction}

A subset~$C$ of a Banach space~$X$ is said to be {\em Banach-Saks} if every sequence $(x_n)_{n}$ in~$C$
admits a Ces\`{a}ro convergent subsequence $(x_{n_j})_{j}$, i.e. the sequence 
of arithmetic means $(\frac{1}{k}\sum_{j=1}^k x_{n_j})_{k}$ is convergent (in the norm topology) to some element of~$X$. 
A Banach space~$X$ is said to have the {\em Banach-Saks property} if its closed unit ball $B_X$
is a Banach-Saks set. An operator $T:Y \to X$ between Banach spaces is said to be {\em Banach-Saks} if so is $T(B_Y)$. 
Every Banach-Saks set is relatively weakly compact (see e.g. \cite[Proposition~2.3]{lop-alt}) and so every space having the Banach-Saks property
is reflexive~\cite{nis-wat}, and every Banach-Saks operator is weakly compact. The converse statements are not true in general~\cite{bae}. 
Every super-reflexive space (like $L^p(\nu)$ for a non-negative measure~$\nu$ and $1<p<\infty$) has the Banach-Saks property (see e.g.
\cite[p.~124]{die-J}). For any non-negative measure~$\nu$, the space $L^1(\nu)$
enjoys the {\em weak Banach-Saks property}, that is, every weakly compact subset of~$L^1(\nu)$ is Banach-Saks, a result due to Szlenk~\cite{szl} (cf. \cite[p.~112]{die-J}). 
At this point it is convenient to recall the Erdös-Magidor theorem~\cite{erd-mag} (cf. \cite[Corollary~2.6]{lop-alt} and \cite[Theorem~2.1]{pou}) which implies,
in particular, that every sequence in a Banach-Saks set admits a subsequence such that all of its subsequences
are Ces\`{a}ro convergent to the same limit:

\begin{theo}[Erdös-Magidor]\label{theo:ErdosMagidor}
Every bounded sequence $(x_n)_n$ in a Banach space~$X$ admits a subsequence~$(x_{n_j})_j$ such that
\begin{enumerate}
\item[(i)] either all subsequences of $(x_{n_j})_j$ are Ces\`{a}ro convergent (to the same limit); 
\item[(ii)] or no subsequence of $(x_{n_j})_j$ is Ces\`{a}ro convergent.
\end{enumerate}
\end{theo}

As we will see, for a {\em finite} measure~$\nu$, the weak Banach-Saks property of~$L^1(\nu)$ yields a somehow similar property for its dual 
$L^1(\nu)^*=L^\infty(\nu)$ by considering the $w^*$-topology and the Mackey topology, namely:
{\em every $w^*$-null sequence in~$L^\infty(\nu)$ admits a subsequence such that all of its subsequences are Ces\`{a}ro convergent to~$0$ 
with respect to $\mu(L^\infty(\nu),L^1(\nu))$.} Recall that, for an arbitrary
Banach space~$X$, the Mackey topology $\mu(X^*,X)$ is the (locally convex) topology on~$X^*$ of uniform convergence 
on all weakly compact subsets of~$X$.
So, for any finite measure~$\nu$, the space $L^1(\nu)$ satisfies the following property which is the main object of study of this paper:

\begin{defi}\label{defi:StarIntro}
A Banach space $X$ is said to have {\em property $(\mu^s)$} if every $w^*$-null sequence in~$X^*$
admits a subsequence such that all of its subsequences are Ces\`{a}ro convergent to~$0$ with respect to $\mu(X^*,X)$.
\end{defi}

The paper is organized as follows. In the preliminary Section~\ref{section:Preliminaries} we point out that property~$(\mu^s)$
is stronger than the so-called property~(K) invented by Kwapie\'{n} in connection with
some results of Kalton and Pe\l cz\'{y}nski~\cite{kal-pel}:

\begin{defi}\label{defi:K}
A Banach space $X$ is said to have {\em property~(K)} if every $w^*$-null sequence in~$X^*$
admits a convex block subsequence which converges to~$0$ with respect to $\mu(X^*,X)$.
\end{defi}

Property~(K) (and some variants)
have been also studied by Frankiewicz and Plebanek~\cite{fra-ple}, 
Figiel, Johnson and Pe\l cz\'{y}nski~\cite{fig-alt2}, de Pagter, Dodds and Sukochev~\cite{dep-dod-suk}, Avil\'{e}s and the author~\cite{avi-rod}.
In Section~\ref{section:Preliminaries} we also give some basic examples of Banach spaces having property~$(\mu^s)$.
For a reflexive space~$X$, $(\mu^s)$ is equivalent to the Banach-Saks property of~$X^*$ (Proposition~\ref{pro:Reflexive}). In particular, any 
super-reflexive space has~$(\mu^s)$. For a $C(L)$ space (where $L$ is a compact Hausdorff topological space), 
$(\mu^s)$ is equivalent to the fact that $C(L)$ is Grothendieck (Proposition~\ref{pro:CK}). So, for instance,
$\ell_\infty$ has property~$(\mu^s)$. 

In Section~\ref{section:StronglyGenerated} we discuss the role of ``strong generation'' in the study of property~$(\mu^s)$. To be more precise
we need some terminology:

\begin{defi}\label{defi:SGfamilies}
Let $X$ be a Banach space and let $\mathcal{H}$ and $\mathcal{G}$ be two families of subsets of~$X$. We say that 
$\mathcal{H}$ is {\em strongly generated} by~$\mathcal{G}$ if for every $H\in\mathcal{H}$ 
and every $\varepsilon>0$ there is $G\in \mathcal{G}$ such that $H \sub G + \varepsilon B_X$.
If in addition $\mathcal{G}=\{nG_0:n\in \Nat\}$ for some $G_0\sub X$, we simply say
that $\mathcal{H}$ is {\em strongly generated} by~$G_0$.
\end{defi}

We will be mainly interested in the case $\mathcal{H}=wk(X)$, the family of all weakly compact subsets of the Banach space~$X$.

\begin{defi}\label{defi:SGoperator}
Let $X$ be a Banach space and let $\mathcal{G}$ be a family of subsets of~$X$. We say that $X$ is {\em strongly generated by~$\mathcal{G}$}
if $wk(X)$ is strongly generated by~$\mathcal{G}$. If in addition
$\mathcal{G}=\{nT(B_Y):n\in \Nat\}$ for some operator $T:Y \to X$ from a Banach space~$Y$, we say that $Y$ {\em strongly generates}~$X$
or that $T$ {\em strongly generates}~$X$.
\end{defi}

Banach spaces which are strongly generated by a reflexive space (i.e. SWCG spaces) 
or by a super-reflexive space have been widely studied, 
see e.g. \cite{fab-mon-ziz,kam-mer2,mer-sta-2} and the references therein.
All SWCG spaces and their subspaces have property~(K), see \cite[Corollary~2.3]{avi-rod}. 
We show that property~$(\mu^s)$ is enjoyed by every subspace of a Banach space which is strongly generated by an operator 
with Banach-Saks adjoint (Theorem~\ref{theo:Generation}).  
This assumption is satisfied by the so-called strongly super weakly compactly generated spaces (S$^2$WCG) studied 
recently by Raja~\cite{raj9} and Cheng et al.~\cite{che2}. In particular, any Banach space which is strongly generated by a super-reflexive space 
(e.g. $L^1(\nu)$ for a finite measure~$\nu$) has property~$(\mu^s)$. 

We prove that a SWCG space~$X$ has property $(\mu^s)$ if (and only if) every $w^*$-null sequence in~$X^*$
admits a subsequence which is Ces\`{a}ro convergent to~$0$ with respect to $\mu(X^*,X)$. 
We do not know whether such equivalence holds for arbitrary Banach spaces. The case of SWCG spaces
is generalized to Banach spaces which are strongly generated by less than~$\mathfrak{p}$ weakly compact sets (Theorem~\ref{theo:p}).
Recall that~$\mathfrak{p}$ is the least cardinality of a family $\mathcal{M}$ of infinite
subsets of~$\N$ such that:
\begin{itemize}
\item $\bigcap \mathcal{N}$ is infinite for every finite subfamily $\mathcal{N} \sub \mathcal{M}$.
\item There is no infinite set $A \sub \N$ such that $A \setminus M$ is finite for all $M\in \mathcal{M}$.
\end{itemize} 
In general, $\omega_1 \leq \mathfrak{p} \leq \mathfrak{c}$. Under CH cardinality less than~$\mathfrak{p}$ just means countable, 
but in other models there are uncountable sets of cardinality less than~$\mathfrak{p}$, 
see e.g.~\cite{bla-J} for more information. 

In Section~\ref{section:sums} we study the stability of property~$(\mu^s)$ under $\ell^p$-sums for $1\leq p \leq \infty$. 
Pe\l cz\'{y}nski showed that the $\ell^1$-sum of $\mathfrak{c}$ copies of $L^1[0,1]$ fails property~(K), see \cite[Example~4.I]{fig-alt2}
(cf.~\cite{fra-ple}). In particular, this implies that property~$(\mu^s)$ is not preserved by arbitrary $\ell^1$-sums. 
We prove that~$(\mu^s)$ is preserved by $\ell^1$-sums of less than~$\mathfrak{p}$ summands (Theorem~\ref{theo:l1sums}),
as well as by arbitrary $\ell^p$-sums whenever $1<p<\infty$ (Theorem~\ref{theo:lpsums}). On the other hand, 
in Example~\ref{exa:Johnson} we point out the existence of a sequence of finite-dimensional spaces
whose $\ell^\infty$-sum fails property~(K), which answers a question left open in \cite[Problem~2.19]{avi-rod}.

In Section~\ref{section:families} we consider a natural weakening of properties~$(\mu^s)$ and~(K)
to deal with certain families of relatively weakly compact sets. This idea is applied to some Banach lattices
and Lebesgue-Bochner spaces. 

\begin{defi}
Let $X$ be a Banach space and let $\mathcal{A}$ be a family of subsets of~$X$. We say that $X$ has
\begin{enumerate}
\item[(i)] property $(\mu^s_\mathcal{A})$ if every $w^*$-null sequence in~$X^*$
admits a subsequence such that all of its subsequences are Ces\`{a}ro convergent to~$0$ uniformly on each element of~$\mathcal{A}$;
\item[(ii)] property (K$_\mathcal{A}$) if every $w^*$-null sequence in~$X^*$
admits a convex block subsequence which converges to~$0$ uniformly on each element of~$\mathcal{A}$.
\end{enumerate}
\end{defi}

For instance, the so-called {\em property~($k$)} of Figiel, Johnson and Pe\l cz\'{y}nski~\cite{fig-alt2}
coincides with property~(K$_{\mathcal{A}}$) when $\mathcal{A}$ is the family 
\begin{equation}\label{eqn:FJP}
	\big\{T(C): \, T:L^1[0,1]\to X \mbox{ operator, } C \in wk(L^1[0,1])\big\},
\end{equation}
see \cite[Lemma~8.1]{dep-dod-suk}. Every weakly sequentially complete Banach lattice with weak unit 
has property~($k$), see \cite[Proposition~4.5(b)]{fig-alt2}. 
This can also be obtained as a consequence of a result of de Pagter, Dodds and Sukochev (see \cite[Theorem~5.3]{dep-dod-suk}) 
stating that every Banach lattice~$X$ with order continuous norm and weak unit has property (K$_{\mathcal{A}}$) when $\mathcal{A}$
is the family of all order bounded subsets of~$X$. We sharpen those results by proving
that, in fact, such Banach lattices have property~$(\mu^s_{\mathcal{A}})$ when $\mathcal{A}$ is the family
described in~\eqref{eqn:FJP} or the family of all $L$-weakly compact sets, respectively (Corollary~\ref{cor:k} and Theorem~\ref{theo:lattice}). 

Finally, we focus on the Lebesgue-Bochner space $L^1(\nu,X)$, where $\nu$ is a finite measure and $X$ is a Banach space.  
It is known that if $X$ contains a subspace isomorphic to~$c_0$, then $L^1([0,1],X)$ contains a {\em complemented} subspace isomorphic to~$c_0$, see~\cite{emm}. 
When applied to the space $\ell^\infty$, this shows that properties $(\mu^s)$ and~(K) do not pass from~$X$ to~$L^1([0,1],X)$ in general (cf. \cite[Remark~6.5]{fig-alt2}).
In fact, in Theorem~\ref{theo:L1X-c0} we prove that $L^1([0,1],X)$ fails property~(K$_{\mathcal{A}}$), 
for the family $\mathcal{A}$ of all $\delta\mathcal{S}$-sets of~$L^1([0,1],X)$, whenever $X$ contains a subspace 
isomorphic to~$c_0$.

\begin{defi}\label{defi:deltaS-set}
A set $K \sub L^1(\nu,X)$ is said to be a $\delta\mathcal{S}$-set if it is uniformly integrable and
for every $\delta>0$ there exists a weakly compact set $W \sub X$ such that
$\nu(f^{-1}(W))\geq 1-\delta$ for every $f\in K$.
\end{defi}

The collection of all $\delta\mathcal{S}$-sets of $L^1(\nu,X)$ will be denoted by $\delta\mathcal{S}(\nu,X)$
or simply $\delta\mathcal{S}$ if no confussion arises.
These sets play an important role when studying weak compactness in Lebesgue-Bochner spaces. 
Any $\delta\mathcal{S}$-set of $L^1(\nu,X)$ is relatively weakly compact, while the converse is not true in general.
For more information on these sets, see \cite{rod17} and the references therein. Concerning positive results, we show that $L^1(\nu,X)$ has property 
$(\mu^s_{\delta\mathcal{S}})$ whenever $X$ is a subspace of a S$^2$WCG space
(Theorem~\ref{theo:L1X-SSRG}). In general, the assumption that $X^*$ has the Banach-Saks property is not enough 
to ensure that $L^1(\nu,X)$ has property~$(\mu^s_{\delta\mathcal{S}})$ (Example~\ref{exa:Schachermayer}).

\section{Notation and preliminaries}\label{section:Preliminaries}

The symbol $|S|$ stands for the cardinality of a set~$S$. All our vector spaces are real. Given a sequence $(f_n)_n$ in a vector space, a {convex block subsequence} of~$(f_n)_n$
is a sequence $(g_k)_k$ of the form
$$
	g_k=\sum_{n\in I_k}a_n f_n
$$
where $(I_k)_k$ is a sequence of finite subsets of~$\N$ with $\max(I_k) < \min(I_{k+1})$ and $(a_n)_n$ is a sequence
of non-negative real numbers such that $\sum_{n\in I_k}a_n=1$ for all~$k \in \N$.
An {\em operator} is a continuous linear map between Banach spaces. 
By a {\em subspace} of a Banach space we mean a closed linear subspace.
Given a Banach space~$X$, its norm is denoted by either $\|\cdot\|_X$ or simply $\|\cdot\|$, and we write $B_X=\{x\in X:\|x\|\leq 1\}$. 
The topological dual of~$X$ is denoted by~$X^*$ and the adjoint of an operator~$T$ is denoted by~$T^*$. The evaluation of $x^*\in X^*$ at $x\in X$ is denoted by either $x^*(x)$
or $\langle x^*,x\rangle$. The weak (resp. weak$^*$) topology on~$X$ (resp.~$X^*$) is denoted by~$w$ (resp.~$w^*$). 

\begin{lem}\label{lem:K}
Let $X$ be a Banach space and let $\mathcal{A}$ be a family of subsets of~$X$. If $X$ has property~$(\mu_\mathcal{A}^s)$, then it also has property~(K$_\mathcal{A}$).
\end{lem}
\begin{proof} Bear in mind that if $(u_n)_n$ is a sequence in a topological vector space which is Ces\`{a}ro convergent to~$0$, then 
it admits a convex block subsequence converging to~$0$. Indeed, define 
$$
	v_k:=\frac{1}{2^{k}}\sum_{n=1}^{2^k}u_n
	\quad \mbox{and} \quad
	w_k:=\frac{1}{2^{k-1}}\sum_{n=2^{k-1}+1}^{2^k}u_n
$$
for every $k\in \Nat$. Then $(w_k)_w$ is a convex block subsequence of~$(u_n)_n$ converging to~$0$, because $(v_k)_k$ converges to~$0$ and 
$v_k=\frac{1}{2}(v_{k-1}+w_k)$ for all $k\geq 2$.
\end{proof}

In particular, {\em property~$(\mu^s)$ implies property~(K)}.
The converse is not true in general, since any reflexive space has property~(K), there are reflexive spaces which
fail the Banach-Saks property (see \cite{bae}) and, moreover, we have:

\begin{pro}\label{pro:Reflexive}
Let $X$ be a Banach space. The following statements are equivalent:
\begin{enumerate}
\item[(i)] $X^*$ has the Banach-Saks property;
\item[(ii)] $X$ is reflexive and has property $(\mu^s)$;
\item[(iii)] $X$ contains no subspace isomorphic to~$\ell^1$ and has property $(\mu^s)$.
\end{enumerate}
\end{pro}
\begin{proof} (i)$\impli$(ii): Clearly, the Banach-Saks property of~$X^*$ implies that $X$ has property~$(\mu^s)$. On the other hand, 
as we mentioned in the introduction, every space with the Banach-Saks property is reflexive.  
 
The implication (ii)$\impli$(iii) is obvious, while (iii)$\impli$(ii) follows from Lemma~\ref{lem:K} and
the fact that any Banach space with property~(K) and without subspaces isomorphic to~$\ell^1$ is reflexive,
see \cite[Theorem~2.1]{avi-rod}. 

Finally, (ii)$\impli$(i) follows from the fact that, if $X$ is reflexive, then $\mu(X^*,X)$ agrees with the norm topology of~$X^*$
and $B_{X^*}$ is $w^*$-sequentially compact.  
\end{proof}

\begin{pro}\label{pro:CK}
Let $L$ be a compact Hausdorff topological space. The following statements are equivalent:
\begin{enumerate}
\item[(i)] $C(L)$ is Grothendieck (i.e. every $w^*$-convergent sequence in~$C(L)^*$ is weakly convergent);
\item[(ii)] $C(L)$ has property $(\mu^s)$.
\end{enumerate} 
\end{pro}
\begin{proof} 
$C(L)^*$ is isomorphic (in fact, order isometric) to the $L^1$-space of a non-negative measure, so it has the weak Banach-Saks property. 
The implication (i)$\impli$(ii) follows at once from this.

(ii)$\impli$(i): Apply Lemma~\ref{lem:K} and the fact that $C(L)$ is Grothendieck if (and only if) it has property~(K), see \cite[Corollary~2.5]{avi-rod}. 
\end{proof}

A subspace~$Y$ of a Banach space~$X$ is said to be {\em $w^*$-extensible} in~$X$  
if every $w^*$-null sequence in~$Y^*$ admits a subsequence which can be extended to a $w^*$-null sequence in~$X^*$. 
It is easy to check that: (i)~any complemented subspace is $w^*$-extensible; and (ii)~if $B_{X^*}$ is $w^*$-sequentially compact, then any subspace is $w^*$-extensible in~$X$
(see \cite[Theorem~2.1]{wan-alt}). The following is straightforward:

\begin{rem}\label{rem:subspaces}
Let $X$ be a Banach space and let $Y\sub X$ be a subspace which is $w^*$-extensible in~$X$. 
Let $\mathcal{A}$ be a family of subsets of~$Y$.
If $X$ has property~$(\mu^s_{\mathcal{A}})$,
then $Y$ has property~$(\mu^s_{\mathcal{A}})$ as well.
\end{rem}

In general, the statement of Remark~\ref{rem:subspaces} is not true for arbitrary subspaces. For instance,
$\ell^\infty$ has property $(\mu^s)$ (by Proposition~\ref{pro:CK} and the fact that $\ell^\infty$ is Grothendieck) but $c_0$ does not
(since it fails~(K), see e.g. \cite[Proposition~C]{kal-pel}).

\section{Strong generation and property~$(\mu^s)$}\label{section:StronglyGenerated}

Our first result in this section provides a sufficient condition for property~$(\mu^s)$.

\begin{theo}\label{theo:Generation}
Let $X$ be a Banach space which is strongly generated by an operator with Banach-Saks adjoint. Then every subspace of~$X$ 
has property $(\mu^s)$. 
\end{theo}

The proof of Theorem~\ref{theo:Generation} requires two lemmas which will be used again later.

\begin{lem}\label{lem:TransferMU-StrongGeneration}
Let $X$ be a Banach space. Let $\mathcal{H}$ and $\mathcal{G}$ be two families of subsets of~$X$ such that 
$\mathcal{H}$ is strongly generated by~$\mathcal{G}$. 
\begin{enumerate}
\item[(i)] If $(x_n^*)_n$ is a bounded sequence in~$X^*$ converging to~$0$
uniformly on each element of~$\cG$, then it also converges to~$0$
uniformly on each element of~$\cH$.  
\item[(ii)] If $X$ has property~$(\mu^s_{\cG})$, then it also has property~$(\mu^s_{\cH})$.
\end{enumerate}
\end{lem}
\begin{proof} (ii) is an immediate consequence of~(i) applied to the corresponding sequences of Ces\`{a}ro means.
For the proof of~(i), let $c>0$ be a constant such that $\|x_n^*\|_{X^*}\leq c$ for all $n\in \Nat$. Fix $H\in \cH$ and take any
$\epsilon>0$. Pick $G\in \cG$ such that $H \sub G+\epsilon B_X$ and choose $n_0\in \Nat$ such that $\sup_{x\in G}|x_n^*(x)|\leq \epsilon$
for all $n\geq n_0$. Then $\sup_{x\in H}|x_n^*(x)|\leq (1+c)\epsilon$ for all $n\geq n_0$.
\end{proof}

\begin{lem}\label{lem:OneSet}
Let $X$ and $Y$ be Banach spaces and let $T:Y \to X$ be an operator such that $T^*$ is Banach-Saks.
Then $X$ has property $(\mu^s_{\{T(B_Y)\}})$.
\end{lem}
\begin{proof}
Let $(x_n^*)_n$ be a $w^*$-null sequence in~$X^*$. Since $T^*:X^*\to Y^*$
is Banach-Saks and $w^*$-$w^*$-continuous, there is a subsequence of~$(x_n^*)$, not relabeled,
such that for every further subsequence $(x_{n_k}^*)_k$ we have that $(T^*(x^*_{n_k}))_k$ is Ces\`{a}ro convergent to~$0$ in norm, that is,
$$
	\lim_{N\to \infty}\sup_{y\in B_Y}\Big|\Big\langle \frac{1}{N}\sum_{k=1}^N x_{n_k}^*,T(y) \Big\rangle\Big|=
	\lim_{N\to \infty}\Big\|\frac{1}{N}\sum_{k=1}^N T^*(x_{n_k}^*)\Big\|_{Y^*}=0.
$$
This shows that $X$ has property $(\mu^s_{\{T(B_Y)\}})$.
\end{proof}

\begin{proof}[Proof of Theorem~\ref{theo:Generation}]
Since any Banach-Saks operator is weakly compact, the space $X$ is SWCG. In particular, $X$ is weakly compactly
generated and so $B_{X^*}$ is $w^*$-sequentially compact (see e.g. \cite[p.~228, Theorem~4]{die-J}). 
Therefore, every subspace is $w^*$-extensible in~$X$ and so it suffices to prove that $X$ has property $(\mu^s)$ 
(see Remark~\ref{rem:subspaces} and the paragraph preceding it).
To this end, let $Y$ be a Banach space and let $T:Y \to X$ be an operator such that $T^*$ is Banach-Saks and $wk(X)$ is strongly generated by~$T(B_Y)$.
By Lemmas~\ref{lem:OneSet} and~\ref{lem:TransferMU-StrongGeneration}(ii), $X$ has property~$(\mu^s)$.
\end{proof}

An operator between Banach spaces $T:Y\to X$ is said to be {\em super weakly compact} 
if the ultrapower $T^\mathcal{U}:Y^\mathcal{U}\to X^\mathcal{U}$ is weakly compact for every free ultrafilter~$\mathcal{U}$ on~$\N$. This
is equivalent to being {\em uniformly convexifying} in the sense of Beauzamy~\cite{bea1}, see e.g. \cite[Theorem~5.1]{hei2}. 
A Banach space $X$ is said to be {\em strongly super weakly compactly generated (S$^2$WCG)} 
if it is strongly generated by a super weakly compact operator (see~\cite{raj9}). In general:
\begin{center}
strongly generated by a super-reflexive space $\Longrightarrow$ S$^2$WCG $\Longrightarrow$ strongly generated by an operator with Banach-Saks adjoint.
\end{center}
The first implication is clear. The second one holds because an operator is super weakly compact if and only if its adjoint is super weakly compact 
(see \cite[Proposition~II.4]{bea1}) 
and any super weakly compact operator is Banach-Saks (see \cite[Th\'{e}or\`{e}me~3]{bea2}).
An example of a S$^2$WCG space which is not strongly generated by a super-reflexive space can be found in \cite[Example~3.10]{raj9}. 
The converse of the second implication above is neither true in general, even for reflexive spaces. Indeed, there are spaces with the Banach-Saks property 
which are not super-reflexive (see \cite{nis-wat}), while every S$^2$WCG reflexive space
is super-reflexive (cf. \cite[Theorem~1.9]{raj9}).

\begin{cor}\label{cor:Generation}
Every subspace of a S$^2$WCG Banach space has property $(\mu^s)$. 
\end{cor}

\begin{cor}\label{cor:L1}
Let $\nu$ be a finite measure. Then every subspace of $L^1(\nu)$ has property~$(\mu^s)$.
\end{cor}

\begin{rem}\label{rem:MercourakisStamati}
The previous results are formulated in terms of ``subspaces'' because the corresponding classes of Banach spaces
are not hereditary. Indeed, Mercourakis and Stamati (see \cite[Theorem~3.9(ii)]{mer-sta-2}) showed that there exist subspaces of 
$L^1[0,1]$ which are not SWCG. 
\end{rem}

The Erdös-Magidor Theorem~\ref{theo:ErdosMagidor} is also valid for Fr\'{e}chet spaces and, under certain set theoretic assumptions, 
for other classes of locally convex spaces, see~\cite{pou}. Along this way, we have the following:

\begin{theo}\label{theo:p}
Let $X$ be a Banach space which is strongly generated by a family $\mathcal{G}\sub wk(X)$
with $|\mathcal{G}|<\mathfrak{p}$ (e.g. a SWCG space). Then $X$ has property $(\mu^s)$ if (and only if) every $w^*$-null sequence in~$X^*$
admits a subsequence which is Ces\`{a}ro convergent to~$0$ with respect to $\mu(X^*,X)$.
\end{theo}

To deal with the proof of Theorem~\ref{theo:p} we need two lemmas. The first one is a standard diagonalization argument.

\begin{lem}\label{lem:diagonalization}
Let $\{\mathcal{S}_\alpha\}_{\alpha<\gamma}$ be a
collection of families of infinite subsets of~$\N$, where $\gamma$ is an ordinal with $\gamma<\mathfrak{p}$. Suppose that, for each $\alpha<\gamma$, we have: 
\begin{enumerate}
\item[(i)] if $B \sub \Nat$ is infinite and $B \setminus A$ is finite for some $A\in \cS_\alpha$, then $B\in \cS_\alpha$;
\item[(ii)] every infinite subset of~$\N$ contains an element of~$\cS_\alpha$.
\end{enumerate}
Then every infinite subset of~$\N$ contains an element of~$\bigcap_{\alpha<\gamma} \cS_\alpha$.
\end{lem}
\begin{proof}
Fix an infinite set $B \sub \N$. We will first construct by transfinite induction a collection $\{B_\alpha:\alpha<\gamma\}$ of subsets of~$B$ with the following properties:
\begin{enumerate}
\item[($P_\alpha$)] $B_\alpha \in \cS_\alpha$ for all $\alpha<\gamma$;
\item[($Q_{\alpha,\beta}$)] $B_\beta \setminus B_\alpha$ is finite whenever $\alpha\leq\beta <\gamma$.
\end{enumerate}
For $\alpha=0$ we just use~(ii) to select any subset $B_0$ of~$B$ belonging to~$\cS_0$.
Suppose now that $1\leq \gamma'< \gamma$ and that we have already constructed 
a collection $\{B_\alpha:\alpha<\gamma'\}$ of subsets of~$B$
such that ($P_\alpha$) and ($Q_{\alpha,\beta}$) hold whenever $\alpha\leq\beta<\gamma'$.
In particular, for any finite set $I \sub \gamma'$ the intersection $\bigcap_{\alpha\in I}B_\alpha$ is infinite. 
Since $\gamma' < \mathfrak{p}$, there is an infinite set~$B_{\gamma'} \sub B$ 
such that $B_{\gamma'} \setminus B_\alpha$ is finite for all $\alpha<\gamma'$.
Property~(i) implies that $B_{\gamma'}\in \bigcap_{\alpha<\gamma'}\cS_\alpha$. Now property~(ii)
implies that, by passing to a further subset of~$B_{\gamma'}$ if necessary, we can assume
that ($P_{\gamma'}$) holds. Clearly, ($Q_{\alpha,\beta}$) also holds for every $\alpha\leq \beta \leq \gamma'$.
This finishes the inductive construction.

Since $\gamma<\mathfrak{p}$ and for any finite set $I \sub \gamma$ the intersection $\bigcap_{\alpha\in I}B_\alpha$ is infinite, 
there is an infinite set $C \sub B$ such that $C \setminus B_\alpha$ is finite for every $\alpha<\gamma$. From~(i) it follows that
$C\in \bigcap_{\alpha<\gamma}\cS_\alpha$.
\end{proof}

The second lemma will also be used in Section~\ref{section:sums}. 

\begin{lem}\label{lem:TVSbelow-p}
Let $\{E_i\}_{i\in I}$ be a family of topological vector spaces with $|I|<\mathfrak{p}$
and let $E:=\prod_{i\in I}E_i$ be equipped with the product topology. For each $i\in I$, we denote
by $\rho_i:E \to E_i$ the $i$th-coordinate projection.
Let $(u_n)_n$ be a sequence in~$E$ satisfying the following condition: 
\begin{itemize}
\item[($\star$)] for every infinite set $A\sub \N$ and every $i\in I$
there is an infinite set $B \sub A$ such that 
the subsequence $(\rho_i(u_n))_{n\in C}$ is Ces\`{a}ro convergent to~$0$ in~$E_i$
for every infinite set $C \sub B$. 
\end{itemize}
Then there is a subsequence of~$(u_n)_n$ such that
 all of its subsequences are Ces\`{a}ro convergent to~$0$ in~$E$.
\end{lem}
\begin{proof}
We will apply Lemma~\ref{lem:diagonalization}. For each $i\in I$, let $\cS_i$
be the family of all infinite sets $A \sub \Nat$ such that for every infinite set $C \sub A$ the
corresponding subsequence $(\rho_i(u_{n}))_{n\in C}$ is Ces\`{a}ro convergent to~$0$ in~$E_i$. 
It suffices to check that conditions (i) and~(ii) of Lemma~\ref{lem:diagonalization} hold 
for this choice. 

Indeed, (ii) follows immediately from ($\star$). On the other hand, fix $i\in I$, $A\in \cS_i$
and an infinite set $B \sub \Nat$ such that $B \setminus A$ is finite. To check that $B\in \cS_i$, take any 
strictly increasing sequence $(n_k)_k$ in~$B$. There is $k_0\in \Nat$ such that $n_k \in A$ for all $k>k_0$, hence 
$(\rho_i(u_{n_k}))_{k>k_0}$ is Ces\`{a}ro convergent to~$0$ in~$E_i$
and so is $(\rho_i(u_{n_k}))_{k}$. 
This shows that $B\in \cS_i$. Thus, condition~(i) of Lemma~\ref{lem:diagonalization} is also satisfied.
\end{proof}

\begin{proof}[Proof of Theorem~\ref{theo:p}]
To prove that $X$ has property~$(\mu^s)$ it suffices
to check that it has property~$(\mu^s_{\cG})$ (Lemma~\ref{lem:TransferMU-StrongGeneration}(ii)).
For each $G\in \cG$, let $R_G:X^*\to C(G)$ be the operator given by $R_G(x^*):=x^*|_G$ (the restriction of~$x^*$ to~$G$).

Let $(x_n^*)_n$ be a $w^*$-null sequence in~$X^*$. Fix an infinite set $A\sub \N$ and $G\in \cG$. 
Since $(R_G(x_n^*))_{n\in A}$ is bounded,
we can apply the Erdös-Magidor Theorem~\ref{theo:ErdosMagidor} to find an infinite set $B \sub A$ such that
either all subsequences of $(R_G(x_{n}^*))_{n\in B}$ are Ces\`{a}ro convergent (to the same limit), or no subsequence of $(R_G(x_{n}^*))_{n\in B}$ is Ces\`{a}ro convergent.
The assumption on~$X$ excludes the second possibility and ensures that all subsequences of $(R_G(x_{n}^*))_{n\in B}$ are Ces\`{a}ro
convergent to~$0$. 

We can now apply Lemma~\ref{lem:TVSbelow-p} to the family of Banach spaces 
$\{C(G)\}_{G\in \cG}$ and the sequence $(u_n)_n$ in~$\prod_{G\in \cG}C(G)$ defined by $u_n:=(R_G(x_n^*))_{G\in \cG}$. So, 
there is a subsequence of~$(x_n^*)_n$ such that all of its subsequences are Ces\`{a}ro convergent to~$0$
uniformly on each $G\in \cG$. This proves that $X$ has property~$(\mu^s_{\cG})$. 
\end{proof}

\section{$\ell^p$-sums}\label{section:sums}

The $\ell^p$-sum ($1\leq p \leq \infty$) of a family of Banach spaces $\{X_i\}_{i\in I}$ is denoted by 
$$
	\Big(\bigoplus_{i\in I}X_i\Big)_{\ell^p}.
$$ 
When $p\neq \infty$ we identify the dual of $(\bigoplus_{i\in I}X_i)_{\ell^p}$ with $(\bigoplus_{i\in I}X^*_i)_{\ell^q}$, 
where $q$ is the conjugate exponent of~$p$, i.e. $1/p+1/q=1$, and
for each $j\in I$ we denote by 
$$
	\pi_j: \Big(\bigoplus_{i\in I}X_i\Big)_{\ell^p}\to X_j \quad\mbox{and}\quad \rho_j:\Big(\bigoplus_{i\in I}X^*_i\Big)_{\ell^q} \to X^*_j
$$ 
the $j$th-coordinate projections.

\begin{lem}\label{lem:l1-convergence-criterion}
Let $\{X_i\}_{i\in I}$ be a family of Banach spaces and $X:=(\bigoplus_{i\in I}X_i)_{\ell^1}$.
Let $(x_n^*)_n$ be a bounded sequence in~$X^*$ such that for every $i\in I$ the sequence $(\rho_i(x_n^*))_n$ is $\mu(X_i^*,X_i)$-null.
Then $(x_n^*)_n$ is $\mu(X^*,X)$-null.
\end{lem}
\begin{proof}
It is known that $wk(X)$ is strongly generated by the family $\mathcal{G}$ consisting of all weakly compact subsets of~$X$ of the form
$$
	\bigcap_{i\in J}\pi_i^{-1}(W_i) \cap \bigcap_{i\in I\setminus J}\pi_i^{-1}(\{0\}),
$$
where $J \sub I$ is finite and $W_i\in wk(X_i)$ for every $i\in J$ (see e.g. \cite[Lemma~7.2(ii)]{kac-alt}).
Clearly, $(x_n^*)_n$ converges to~$0$ uniformly on each element of~$\mathcal{G}$.
From Lemma~\ref{lem:TransferMU-StrongGeneration}(i) it follows that $(x_n^*)_n$ converges to~$0$
with respect to~$\mu(X^*,X)$.
\end{proof}

\begin{theo}\label{theo:l1sums}
Let $\{X_i\}_{i\in I}$ be a family of Banach spaces having property $(\mu^s)$. If $|I|<\mathfrak{p}$,
then $(\bigoplus_{i\in I}X_i)_{\ell^1}$ has property $(\mu^s)$.
\end{theo}
\begin{proof} Write $X:=(\bigoplus_{i\in I}X_i)_{\ell^1}$.
Let $(x_n^*)_n$ be a $w^*$-null sequence in $X^*$. Then $(\rho_i(x_n^*))_n$
is $w^*$-null in~$X_i^*$ for all $i\in I$. Since each $X_i$ has property~$(\mu^s)$, we can apply Lemma~\ref{lem:TVSbelow-p}
to the family of locally convex spaces $\{(X^*_i,\mu(X_i^*,X_i))\}_{i\in I}$
and the sequence $(u_n)_n$ in~$\prod_{i\in I}X^*_i$ defined by $u_n:=(\rho_i(x_n^*))_{i\in I}$. Therefore, there is 
a subsequence of~$(x_n^*)_n$, not relabeled, such that 
for every further subsequence $(x^*_{n_k})_k$ and
every $i\in I$, the sequence $(\rho_i(x^*_{n_k}))_k$ is Ces\`{a}ro convergent to~$0$ with respect to~$\mu(X_i^*,X_i)$.
Now, apply Lemma~\ref{lem:l1-convergence-criterion} to the sequence of arithmetic means of any such subsequence~$(x^*_{n_k})_k$
to conclude that $(x^*_{n_k})_k$ is Ces\`{a}ro convergent to~$0$ with respect to~$\mu(X^*,X)$.
\end{proof}

The following lemma isolates an argument used in the proof of Theorem~2.15 in~\cite{avi-rod}, which says
that property~(K) is preserved by arbitrary~$\ell^p$-sums for $1<p<\infty$. We will use it to prove the same
statement for property~$(\mu^s)$ (Theorem~\ref{theo:lpsums} below).

\begin{lem}\label{lem:lpsums}
Let $\{X_i\}_{i\in I}$ be a family of Banach spaces, let $1<p<\infty$ and write $X:=(\bigoplus_{i\in I}X_i)_{\ell^p}$. Let 
$(y_j^*)_j$ be a sequence in~$X^*$ such that:
\begin{enumerate}
\item[(i)] for every $i\in I$ the sequence $(\rho_i(y^*_{j}))_j$ is $\mu(X_i^*,X_i)$-null;
\item[(ii)] there is a norm convergent sequence $(\tilde{\xi}_j)_j$ in~$\ell^q(I)$ such that 
$$
	\tilde{\xi}_j\geq \xi_j :=(\|\rho_i(y_j^*)\|_{X_i^*})_{i\in I}
	\quad\mbox{pointwise in }\ell^q(I)
$$
for all $j\in \N$. 
\end{enumerate}
Then $(y_j^*)_j$ is $\mu(X^*,X)$-null.
\end{lem}
\begin{proof} Fix any weakly compact set $L \sub B_X$ and $\epsilon>0$. 
Since $(\tilde{\xi}_j)_j$ is norm convergent in~$\ell^q(I)$, there exist a finite set $I_0 \sub I$ and $j_0\in \N$ such that
\begin{equation}\label{eqn:cola-w}
	\sup_{j> j_0} \Big(\sum_{i\in I\setminus I_0} \psi_i(\tilde{\xi}_j)^q\Big)^{\frac{1}{q}} \leq \epsilon,
\end{equation}
where $\psi_i\in \ell^q(I)^*$ denotes the $i$th-coordinate functional. 
Bearing in mind that 
$$
	\|\rho_i(y_j^*)\|_{X_i^*} =\psi_i(\xi_j) \leq \psi_i(\tilde{\xi}_j)
	\quad\mbox{for every }i\in I \mbox{ and }j\in \Nat,
$$
from~\eqref{eqn:cola-w} we get
$$	
	\sup_{j> j_0} \Big(\sum_{i\in I\setminus I_0} \|\rho_i(y_j^*)\|_{X_i^*}^q\Big)^{\frac{1}{q}} \leq \epsilon.
$$
The previous inequality and H\"{o}lder's one imply that
\begin{equation}\label{eqn:cola-ww}
	\sup_{j> j_0} \, \sum_{i\in I \setminus I_0} |b_i|\cdot \|\rho_i(y_j^*)\|_{X_i^*} \leq  \epsilon
	\quad \mbox{for every }(b_i)_{i\in I}\in B_{\ell^p(I)}.
\end{equation}

For each $i\in I$, the sequence $(\rho_i(y_j^*))_j$ converges to~$0$ uniformly on the weakly compact set $\pi_i(L) \sub X_i$. 
Thus, we can find $j_1>j_0$ such that
\begin{equation}\label{eqn:yk-coordinates}
	 \big|\langle \rho_i(y_j^*),\pi_i(x) \rangle \big| \leq \frac{\epsilon}{|I_0|}
	 \quad\mbox{for every }j>j_1, \, i\in I_0 \mbox{ and }x\in L.
\end{equation}
Therefore, for every $j>j_1$ and $x\in L \sub B_X$ we have
\begin{multline*}
	\big|\langle y_j^*,x \rangle \big| \leq
	\sum_{i\in I} \big|\langle \rho_i(y_j^*),\pi_i(x) \rangle \big| \stackrel{\eqref{eqn:yk-coordinates}}{\leq} 
	\epsilon + \sum_{i\in I\setminus I_0} \big|\langle \rho_i(y_j^*),\pi_i(x) \rangle \big| \\
	\leq \epsilon + \sum_{i\in I\setminus I_0} \|\pi_i(x)\|_{X_i} \cdot \|\rho_i(y_j^*)\|_{X_i^*}  \stackrel{\eqref{eqn:cola-ww}}{\leq}
	2\epsilon.
\end{multline*}
This shows that $(y_j^*)_j$ is $\mu(X^*,X)$-null.
\end{proof}

\begin{theo}\label{theo:lpsums}
Let $\{X_i\}_{i\in I}$ be a family of Banach spaces having property~$(\mu^s)$ and let $1<p<\infty$.
Then $(\bigoplus_{i\in I}X_i)_{\ell^p}$ has property $(\mu^s)$.
\end{theo}
\begin{proof}
Write $X:=(\bigoplus_{i\in I}X_i)_{\ell^p}$. Let 
$(x_n^*)_n$ be a $w^*$-null sequence in~$X^*$. 
Define 
$$
	v_n:=(\|\rho_i(x_n^*)\|_{X_i^*})_{i\in I}\in \ell^q(I)
	\quad\mbox{for all }n\in \Nat,
$$ 
so that $\|v_n\|_{\ell^q(I)}=\|x_n^*\|_{X^*}$.
Since $(v_n)_n$ is bounded and~$\ell^q(I)$ has the Banach-Saks property, there exist 
a subsequence of~$(x_n^*)_n$, not relabeled, such that all subsequences of $(v_n)_n$ are 
Ces\`{a}ro convergent in norm (to the same limit).

On the other hand, since every element of~$X^*=(\bigoplus_{i\in I}X^*_i)_{\ell^q}$ is countably supported, we can assume without loss of generality
that $I$ is countable. Then, as in the proof of Theorem~\ref{theo:l1sums}, we can find a subsequence
of~$(x_n^*)_n$, not relabeled, such that for every further subsequence $(x_{n_k}^*)_k$ and
every $i\in I$, the sequence $(\rho_i(x^*_{n_k}))_k$ is Ces\`{a}ro convergent to~$0$ with respect to~$\mu(X_i^*,X_i)$.

We claim that any subsequence $(x_{n_k}^*)_k$ is Ces\`{a}ro convergent to~$0$ with respect to $\mu(X^*,X)$.
Indeed, define $y_j^*:=\frac{1}{j}\sum_{k=1}^j x_{n_k}^*$ for all $j\in \N$. 
We will show that $(y^*_j)_{j}$ is $\mu(X^*,X)$-null by checking that it satisfies conditions~(i) and~(ii) of Lemma~\ref{lem:lpsums}.
Obviously, (i) holds. On the other hand, for each $i\in I$ and $j\in \N$ we have
$$
	\|\rho_i(y_j^*)\|_{X_i^*} \leq \frac{1}{j}\sum_{k=1}^j\|\rho_i(x_{n_k}^*)\|_{X^*_i}=
	\psi_i\Big(\frac{1}{j}\sum_{k=1}^j v_{n_k}\Big),
$$
where $\psi_i\in \ell^q(I)^*$ denotes the $i$th-coordinate functional. Hence, condition~(ii)
of Lemma~\ref{lem:lpsums} holds by taking $\tilde{\xi}_j:=\frac{1}{j}\sum_{k=1}^j v_{n_k}$
for all $j\in \N$.
\end{proof}

\begin{rem}\label{rem:MUnoSWCG}
The previous result provides examples of separable Banach spaces having property~$(\mu^s)$ which 
do not embed isomorphically into any SWCG space, like $\ell^p(\ell^1)$ and $\ell^p(L^1[0,1])$ for $1<p<\infty$, see
\cite[Corollary~2.29]{kam-mer2}. 
\end{rem}

The following example answers in the negative a question left
open in \cite[Problem~2.19]{avi-rod}. It also shows that property $(\mu^s)$ 
is not preserved by countable $\ell^\infty$-sums.

\begin{exa}\label{exa:Johnson}
There is a sequence $(X_n)_n$ of finite-dimensional Banach spaces such that $(\bigoplus_{n\in \Nat}X_n)_{\ell^\infty}$ fails property~(K).
\end{exa}
\begin{proof}
Johnson~\cite{joh-2} proved the existence of a sequence $(X_n)_n$ of finite-dimensional Banach spaces such that, for every separable Banach space~$X$,
its dual $X^*$ is isomorphic to a {\em complemented} subspace of $(\bigoplus_{n\in \Nat}X_n)_{\ell^\infty}$.

Bearing in mind that property~(K) is inherited by complemented subspaces, the fact that the space
$(\bigoplus_{n\in \Nat}X_n)_{\ell^\infty}$ fails property~(K) follows from the existence of separable Banach spaces
whose dual fails property~(K), like $C[0,1]$ and the predual of the James tree space~$JT$. Indeed,
$C[0,1]^*$ is isomorphic to the $\ell^1$-sum of $\mathfrak{c}$ many copies of~$L^1[0,1]$
(see e.g. \cite[p.~242, Remark~5]{ros-J-4}) and so it fails property~(K), according to Pe\l czy\'{n}ski's example
mentioned in the introduction. On the other hand, $JT$ is not reflexive
and contains no subspace isomorphic to~$\ell^1$, hence $JT$ fails property~(K), see \cite[Theorem~2.1]{avi-rod}.
\end{proof}

\section{Applications}\label{section:families}

\subsection{Banach lattices}\label{subsection:lattices}

Given a Banach lattice~$X$, we write
$$
	L_w(X):=\{A \sub X: \, A\mbox{ is }L\mbox{-weakly compact}\}.
$$
Recall that a bounded set $A \sub X$ is said to be {\em $L$-weakly compact} if every disjoint sequence contained in 
$\bigcup_{x\in A}[-|x|,|x|]$ (the solid hull of~$A$) is norm null. 
Every $L$-weakly compact set is relatively weakly compact, while the converse does not hold in general.
$L$-weak compactness and relative weak compactness are equivalent
for subsets of the $L^1$-space of a non-negative measure. 
More generally, if $X$ has order continuous norm, then a set $A \sub X$ is $L$-weakly compact if and only if 
it is {\em approximately order bounded}, i.e. for every $\epsilon>0$ there is $x\in X^+$ such that $A \sub [-x,x]+\epsilon B_{X}$. For more information on
$L$-weakly compact sets, see \cite[\S 3.6]{mey2}.  

\begin{theo}\label{theo:lattice}
Let $X$ be a Banach lattice with order continuous norm and weak unit. Then $X$ has property~$(\mu^s_{L_w(X)})$.
\end{theo}
\begin{proof} Such a Banach lattice is order isometric to a Köthe function space over a finite measure space, see e.g. \cite[Theorem~1.b.14]{lin-tza-2}. 
So, we can assume that $X$ is a Köthe function space over a finite measure space, say~$(\Omega,\Sigma,\nu)$. 
Let $i:L^\infty(\nu)\to X$ be the inclusion operator. Since $X$ has order continuous norm, $i^*(X^*)\sub L^1(\nu)$ 
and so $i^*: X^*\to L^1(\nu)$ is $w^*$-$w$-continuous (see e.g. \cite[p.~29]{lin-tza-2}).

The order continuity of the norm also ensures that $L_w(X)$ is strongly generated by $i(B_{L^\infty(\nu)})$ 
(see e.g. \cite[Lemma~2.37(iii)]{oka-alt}). Therefore,
in order to prove that $X$ has property~$(\mu^s_{L_w(X)})$ we only have to check that $X$ has property~$(\mu^s_{\{i(B_{L^\infty(\nu)})\}})$
(by Lemma~\ref{lem:TransferMU-StrongGeneration}(ii)). 
To this end, let $(x^*_n)_n$ be a $w^*$-null sequence in~$X^*$. Then $(i^*(x^*_n))_n$ is weakly null in~$L^1(\nu)$, which has the weak Banach-Saks property. 
Hence there is a subsequence of~$(x^*_n)_n$, not relabeled, such that any further subsequence $(i^*(x^*_{n_k}))_k$ is Ces\`{a}ro convergent 
to~$0$ in the norm of~$L^1(\nu)$. In particular,
$$
	\lim_{N\to \infty}\sup_{f\in B_{L^\infty(\nu)}}\Big|\Big\langle \frac{1}{N}\sum_{k=1}^N x^*_{n_k},i(f)\Big\rangle\Big|
	= \lim_{N\to \infty}\Big\|\frac{1}{N}\sum_{k=1}^N i^*(x^*_{n_k})\Big\|_{L^1(\nu)}=0.
$$
This shows that $X$ has property~$(\mu^s_{\{i(B_{L^\infty(\nu)}\})})$.
\end{proof}

If $X$ is a weakly sequentially complete Banach lattice, then
$$
	L_w(X) \supseteq \mathcal{Q}(X):=\big\{T(C): \, T:L^1[0,1]\to X \mbox{ operator, } C \in wk(L^1[0,1])\big\}.
$$ 
Indeed, in this case $X$ has order continuous norm (see e.g. \cite[Theorems~2.4.2 and~2.5.6]{mey2}) and any operator $T:L^1[0,1]\to X$ is regular
(see e.g. \cite[Theorem~1.5.11]{mey2}), so it maps approximately order bounded (i.e. relatively weakly compact) subsets
of~$L^1[0,1]$ to approximately order bounded (i.e. $L$-weakly compact) subsets of~$X$. Thus, we get the following:

\begin{cor}\label{cor:k}
Let $X$ be a weakly sequentially complete Banach lattice with weak unit. Then $X$ has property~$(\mu^s_{\mathcal{Q}(X)})$.
\end{cor}

A Banach lattice is said to have the {\em positive Schur property (PSP)} if every weakly null sequence
of positive vectors is norm null. Obviously, this property is satisfied by the $L^1$-space of any non-negative measure. The PSP 
implies weak sequential completeness (see e.g. \cite[Theorem~2.5.6]{mey2}) and so the order continuity of the norm. 
On the other hand, it is known that the PSP is equivalent to saying that every relatively weakly compact set
is $L$-weakly compact (see e.g. \cite[Theorem~3.14]{gao-xan}).  So, Theorem~\ref{theo:lattice} implies that {\em every Banach lattice with the PSP and weak unit
has property~$(\mu^s)$}. In fact, such a Banach lattice is S$^2$WCG, as we show in Corollary~\ref{cor:PSP} below. The key is the following
lemma (cf. \cite[Corollary~5.6]{flo-her-ray}):

\begin{lem}\label{lem:Kothe}
Let $X$ be a Köthe function space with order continuous norm over a finite measure space $(\Omega,\Sigma,\nu)$. Then the inclusion 
operator $i:L^\infty(\nu) \to X$ is super weakly compact.
\end{lem}
\begin{proof}
Since $X$ has order continuous norm, every order interval of~$X$ is weakly compact (see e.g. \cite[Theorem~2.4.2]{mey2}). Hence
$i$ is weakly compact and so is $i^*: X^*\to L^1(\nu)$ (see the proof of Theorem~\ref{theo:lattice}). Every weakly compact operator taking values in
the $L^1$-space of a finite measure is super weakly compact, see e.g. \cite[p.~123, Remarque]{bea1} (cf. \cite[Proposition~5.5]{flo-her-ray}). Therefore,
$i^*$ is super weakly compact and the same holds for~$i$ (see \cite[Proposition~II.4]{bea1}). 
\end{proof}

\begin{cor}\label{cor:PSP}
Let $X$ be a Banach lattice with the positive Schur property and weak unit. Then $X$ is S$^2$WCG.
\end{cor}
\begin{proof}
As in Theorem~\ref{theo:lattice}, we can assume that $X$ is a Köthe function space over a finite measure space
$(\Omega,\Sigma,\nu)$. Then $X$ is strongly generated by the inclusion operator $i:L^\infty(\nu) \to X$
(bear in mind that every weakly compact subset of~$X$ is $L$-weakly compact). 
On the other hand, $i$ is super weakly compact by Lemma~\ref{lem:Kothe}. 
\end{proof}

The previous corollary is an improvement of \cite[Proposition~5.6]{cal-alt-6}, where it was shown that
such Banach lattices are SWCG.

\begin{rem}\label{rem:weakunit}
The conclusion of Theorem~\ref{theo:lattice} and Corollaries~\ref{cor:k} and~\ref{cor:PSP} can fail in the
absence of weak unit. For instance, let $X$ be the $\ell^1$-sum of $\mathfrak{c}$ many copies of~$L^1[0,1]$. 
Then $X$ has the PSP and fails property~$(\mu_{\mathcal{Q}(X)}^s)$, since it does not have 
property~(k) (see \cite[Example~4.I]{fig-alt2}).
\end{rem}

\subsection{Lebesgue-Bochner spaces}\label{subsection:Lebesgue-Bochner}

The first result of this subsection is based on the proof of Emmanuele's result~\cite{emm} on complemented copies of~$c_0$
in Lebesgue-Bochner spaces, cf. \cite[Theorem~4.3.2]{cem-men}.

\begin{theo}\label{theo:L1X-c0}
Let $X$ be a Banach space containing a subspace isomorphic to~$c_0$. Then $L^1([0,1],X)$ fails property (K$_{\delta\mathcal{S}}$).
\end{theo}

For the proof of Theorem~\ref{theo:L1X-c0} we need a lemma.

\begin{lem}\label{lem:c0}
Let $X$ be a Banach space containing a $c_0$-sequence $(x_n)_n$.
Let $(h_n)_n$ be a sequence of $\{-1,1\}$-valued measurable functions on~$[0,1]$
and let $(I_k)_k$ be a sequence of finite subsets of~$\Nat$ with $\max(I_k) < \min(I_{k+1})$
for all $k\in \Nat$. Then 
$$
	\Big\{\sum_{n\in I_k}h_n(\cdot) x_n: \, k\in \Nat\Big\}
$$ 
is a $\delta\cS$-set of $L^1([0,1],X)$.
\end{lem}
\begin{proof}
It suffices to show that the set
$$
	C:=\Big\{\sum_{n\in I_k}a_n x_n:\, k\in \Nat, \, (a_n)_{n\in I_k} \in \{-1,1\}^{I_k}\Big\}
$$ 
is relatively weakly compact in~$X$. To this end, let $(y_m)_m$ be a sequence in~$C$.
For each $m\in \Nat$ we write
$$
	y_m=\sum_{n\in I_{k_m}}a_{n,m} x_n
$$
for some $k_m\in \Nat$ and $a_{n,m}\in \{-1,1\}$. By passing to a subsequence, not relabeled, we
can assume that one of the following alternatives holds:
\begin{itemize}
\item There is $k\in \Nat$ such that $k_m=k$ for all $m\in \Nat$. In this case,
$(y_m)_m$ is a bounded sequence in a finite-dimensional subspace of~$X$ and, therefore, it admits
a norm convergent subsequence.
\item $k_m<k_{m+1}$ for all $m\in \Nat$. In this case, since $(x_n)_n$ is a $c_0$-sequence, 
the same holds for~$(y_m)_m$ and so it is weakly null.
\end{itemize}
Thus $C$ is relatively weakly compact, as required.
\end{proof}

\begin{proof}[Proof of Theorem~\ref{theo:L1X-c0}]
We denote by $(e_n)_n$ and $(e^*_n)_n$ the usual bases of~$c_0$ and~$\ell^1$, respectively.
Let $(x_n)_n$ be a $c_0$-sequence in~$X$ and let $(r_n)_n$ be the sequence of Rademacher functions on~$[0,1]$. Then $(r_n(\cdot)x_n)_n$ is a $c_0$-sequence
in $L^1([0,1],X)$ which spans a complemented subspace
$$
	Z:=\overline{{\rm span}}\{r_{n}(\cdot)x_{n}: \, n\in \Nat\} \sub L^1([0,1],X),
$$ 
see e.g. the proof of Theorem~4.3.2 in~\cite{cem-men}. Let $P:L^1([0,1],X)\to Z$
be a projection and let $T:Z \to c_0$ be the isomorphism satisfying $T(r_{n}(\cdot) x_{n})=e_n$ for all $n\in \Nat$. 
Consider the operator $S:=T\circ P:L^1([0,1],X)\to c_0$. Note that $(S^*(e_n^*))_n$ is a
$w^*$-null sequence in~$L^1([0,1],X)^*$.

{\em Claim.} $(S^*(e_n^*))_n$ does not admit convex block subsequences
converging to~$0$ uniformly on each $\delta\cS$-set. 
Indeed, let $(g_k)_k$ be any convex block subsequence of~$(S^*(e_n^*))_n$. Write $g_k=\sum_{n\in I_k}a_n S^*(e_n^*)$, where
$(I_k)_k$ is a sequence of finite subsets of~$\Nat$ with $\max(I_k) < \min(I_{k+1})$ and $a_n\geq 0$
satisfy $\sum_{n\in I_k}a_n=1$. 
Then for each $k\in \Nat$ we have
\begin{multline*}
	\Big\langle g_k, \sum_{n\in I_k}r_{n}(\cdot) x_{n} \Big\rangle=
	\Big\langle \sum_{n\in I_k}a_n e_n^*, \sum_{n\in I_k}S(r_{n}(\cdot) x_{n}) \Big\rangle \\ =
	\Big\langle \sum_{n\in I_k}a_n e_n^*, \sum_{n\in I_k}e_n \Big\rangle =
	\sum_{n\in I_k}a_n=1.
\end{multline*}
Hence $(g_k)_k$ does not converge to~$0$ uniformly on the $\delta\cS$-set 
$$
	\Big\{\sum_{n\in I_k}r_{n}(\cdot) x_{n}: \, k\in \Nat\Big\}
$$ 
(Lemma~\ref{lem:c0}). This proves that $L^1([0,1],X)$ fails property (K$_{\delta\mathcal{S}}$).
\end{proof}

From now on $(\Omega,\Sigma,\nu)$ is a finite measure space. Given a Banach space~$X$, the identity operator
$i: L^2(\nu,X) \to L^1(\nu,X)$
strongly generates $L^1(\nu,X)$. Indeed, this can be checked as in the case of real-valued functions, bearing in mind that
any weakly compact subset of $L^1(\nu,X)$ is uniformly integrable (see e.g. \cite[p.~104, Theorem~4]{die-uhl-J}).
On the other hand, $L^2(\nu,X)$ is super-reflexive whenever $X$ is super-reflexive, see \cite{day} (cf. \cite[Ch.~IV, Corollary~4.5]{dev-alt-J}). 
In particular, $L^1(\nu,X)$ is strongly generated by a super-reflexive space
(and so it has property~$(\mu^s)$, by Corollary~\ref{cor:Generation}) whenever $X$ is super-reflexive.

\begin{theo}\label{theo:L1X-SSRG}
Let $X$ be a S$^2$WCG Banach space and let $Z \sub X$ be a subspace. Then $L^1(\nu,Z)$ has property~$(\mu^s_{\delta\mathcal{S}(\nu,Z)})$.
\end{theo}
\begin{proof} 
The space $L^1(\nu,X)$ is weakly compactly generated because $X$ is
(see e.g. \cite[p.~252, Corollary~11]{die-uhl-J}). Therefore,
as in the proof of Theorem~\ref{theo:Generation}, it suffices to check that $L^1(\nu,X)$ has property~$(\mu^s_{\delta\mathcal{S}(\nu,X)})$.

Let $Y$ be a Banach space which strongly generates~$X$ through a super weakly compact operator $T:Y \to X$.
We can assume that $Y$ is reflexive and $T$ is injective, according to Theorem~4.5 and Proposition~4.6 in~\cite{raj10}.
Then $\delta\mathcal{S}(\nu,X)$ is strongly generated by 
$$
	H:=\{h\in L^1(\nu,X): \, h(\omega)\in T(B_Y) \mbox{ for }\mu\mbox{-a.e. }\omega\in \Omega\}
$$
(see \cite[proof of Theorem~2.7]{rod13}). 
Let $\tilde{T}:L^2(\nu,Y)\to L^2(\nu,X)$ be 
the ``composition'' operator defined by 
$$
	\tilde{T}(f):=T\circ f
	\quad\mbox{for all }f\in L^2(\nu,Y).
$$ 
Let $i:L^2(\nu,X)\to L^1(\nu,X)$ be the identity operator and define $S:=i\circ \tilde{T}$.
Since $Y$ is reflexive and $T$ is injective, we have $H \sub S(\rho B_{L^2(\nu,Y)})$ for $\rho:=(\nu(\Omega))^{1/2}$ (see \cite[proof of Theorem~1]{laj-rod}),
hence $\delta\mathcal{S}(\nu,X)$ is strongly generated by $S(B_{L^2(\nu,Y)})$. 

Since $T$ is super weakly compact,
so is $\tilde{T}$ (see \cite[p.~126, Corollaire]{bea1}), hence
$S$ is super weakly compact as well. Therefore, $S^*$ is super weakly compact
(see \cite[Proposition~II.4]{bea1}) and so
$S^*$ is Banach-Saks (see \cite[Th\'{e}or\`{e}me~3]{bea2}). From Lemmas~\ref{lem:OneSet}
and~\ref{lem:TransferMU-StrongGeneration}(ii) we conclude that $L^1(\nu,X)$
has property~$(\mu^s_{\delta\mathcal{S}(\nu,X)})$.
\end{proof}

Our final example shows that, in the statement of Theorem~\ref{theo:L1X-SSRG}, the S$^2$WCG property  
cannot be replaced by the Banach-Saks property of the dual.
It is known (and not difficult to check) that if $X$ is a reflexive Banach space, then every relatively weakly compact subset of~$L^1(\nu,X)$ is a
$\delta\mathcal{S}$-set, so in this particular setting properties $(\mu^s)$ and~$(\mu^s_{\delta\mathcal{S}})$ are equivalent. 

\begin{exa}\label{exa:Schachermayer}
There exists a Banach space~$X$ such that $X^*$ has the Banach-Saks property but $L^1([0,1],X)$ 
fails property $(\mu^s)$.
\end{exa}
\begin{proof}
Schachermayer~\cite{sch-4} constructed an example of a Banach space~$E$ having the Banach-Saks property
such that $L^2([0,1],E)$ does not have it. The failure of the property is witnessed by a {\em uniformly bounded}
weakly null sequence $(f_n)_n$ in $L^2([0,1],E)$ (see \cite{sch-4}, proof of Proposition~3).

Set $X:=E^*$ and let $i:L^2([0,1],X)\to L^1([0,1],X)$ be the identity operator.  
Since $E$ is reflexive, the same holds for $L^2([0,1],E)$ and we have
$$
	L^2([0,1],E)^*=L^2([0,1],X) \quad\mbox{and}\quad L^1([0,1],X)^*=L^\infty([0,1],E),
$$
see e.g. \cite[IV.1]{die-uhl-J}.
Moreover, $(f_n)_n$ is a $w^*$-null sequence in $L^1([0,1],X)^*$.
No subsequence $(f_{n_k})_k$ is Ces\`{a}ro convergent to~$0$ in the norm of~$L^2([0,1],E)$, so it cannot be
Ces\`{a}ro convergent to~$0$ uniformly on the weakly compact set $i(B_{L^2([0,1],X)})$. This shows that
$L^1([0,1],X)$ fails property $(\mu^s)$.
\end{proof}

\subsection*{Acknowledgements}
This research was supported by projects MTM2014-54182-P and MTM2017-86182-P (AEI/FEDER, UE).

\providecommand{\MR}{\relax\ifhmode\unskip\space\fi MR }

\bibliographystyle{amsplain}

\end{document}